\documentclass[11pt]{amsart}
\usepackage{amssymb,amsmath,epsfig,mathrsfs, enumerate, xparse, mathtools}
\usepackage[pagewise]{lineno}
\usepackage[nodisplayskipstretch]{setspace}
\usepackage{graphicx}\usepackage[normalem]{ulem}
\usepackage{fancyhdr}
\usepackage[margin=2.5cm]{geometry}
\usepackage{hyperref}
\hypersetup{
 colorlinks   = true,
 urlcolor     = blue,
 linkcolor    = blue,
 citecolor   = red ,
 bookmarksopen=true
}

\theoremstyle{plain}
\newtheorem{thrm}{Theorem}[section]
\newtheorem{lemma}[thrm]{Lemma}

\allowdisplaybreaks
\begin{document}
\newcommand{\sn}{\mathbb{S}^{n-1}}
\newcommand{\SL}{\mathcal L^{1,p}( D)}
\newcommand{\Lp}{L^p( Dega)}
\newcommand{\py}{  \partial_{y}^a}
\newcommand{\La}{\mathscr{L}_a}
\newcommand{\CO}{C^\infty_0( \Omega)}
\newcommand{\Rn}{\mathbb R^n}
\newcommand{\Rm}{\mathbb R^m}
\newcommand{\R}{\mathbb R}
\newcommand{\Om}{\Omega}
\newcommand{\Hn}{\mathbb H^n}
\newcommand{\aB}{\alpha B}
\newcommand{\eps}{\ve}
\newcommand{\BVX}{BV_X(\Omega)}
\newcommand{\p}{\partial}
\newcommand{\IO}{\int_\Omega}
\newcommand{\bG}{\boldsymbol{G}}
\newcommand{\bg}{\mathfrak g}
\newcommand{\bz}{\mathfrak z}
\newcommand{\bv}{\mathfrak v}
\newcommand{\Bux}{\mbox{Box}}
\newcommand{\e}{\ve}
\newcommand{\X}{\mathcal X}
\newcommand{\Y}{\mathcal Y}
\newcommand{\W}{\mathcal W}
\newcommand{\la}{\lambda}
\newcommand{\vf}{\varphi}
\newcommand{\rhh}{|\nabla_H \rho|}
\newcommand{\Ba}{\mathcal{B}_\beta}
\newcommand{\Za}{Z_\beta}
\newcommand{\ra}{\rho_\beta}
\newcommand{\n}{\nabla}
\newcommand{\vt}{\vartheta}
\newcommand{\its}{\int_{\{y=0\}}}

\numberwithin{equation}{section}

\newcommand{\RN} {\mathbb{R}^N}
\newcommand{\Sob}{S^{1,p}(\Omega)}
\newcommand{\Dxk}{\frac{\partial}{\partial x_k}}
\newcommand{\Co}{C^\infty_0(\Omega)}
\newcommand{\Je}{J_\ve}
\newcommand{\beq}{\begin{equation}}
\newcommand{\bea}[1]{\begin{array}{#1} }
\newcommand{\eeq}{ \end{equation}}
\newcommand{\ea}{ \end{array}}
\newcommand{\eh}{\ve h}
\newcommand{\Dxi}{\frac{\partial}{\partial x_{i}}}
\newcommand{\Dyi}{\frac{\partial}{\partial y_{i}}}
\newcommand{\Dt}{\frac{\partial}{\partial t}}
\newcommand{\aBa}{(\alpha+1)B}
\newcommand{\GF}{\psi^{1+\frac{1}{2\alpha}}}
\newcommand{\GS}{\psi^{\frac12}}
\newcommand{\HFF}{\frac{\psi}{\rho}}
\newcommand{\HSS}{\frac{\psi}{\rho}}
\newcommand{\HFS}{\rho\psi^{\frac12-\frac{1}{2\alpha}}}
\newcommand{\HSF}{\frac{\psi^{\frac32+\frac{1}{2\alpha}}}{\rho}}
\newcommand{\AF}{\rho}
\newcommand{\AR}{\rho{\psi}^{\frac{1}{2}+\frac{1}{2\alpha}}}
\newcommand{\PF}{\alpha\frac{\psi}{|x|}}
\newcommand{\PS}{\alpha\frac{\psi}{\rho}}
\newcommand{\ds}{\displaystyle}
\newcommand{\Zt}{{\mathcal Z}^{t}}
\newcommand{\XPSI}{2\alpha\psi \begin{pmatrix} \frac{x}{\left< x \right>^2}\\ 0 \end{pmatrix} - 2\alpha\frac{{\psi}^2}{\rho^2}\begin{pmatrix} x \\ (\alpha +1)|x|^{-\alpha}y \end{pmatrix}}
\newcommand{\Z}{ \begin{pmatrix} x \\ (\alpha + 1)|x|^{-\alpha}y \end{pmatrix} }
\newcommand{\ZZ}{ \begin{pmatrix} xx^{t} & (\alpha + 1)|x|^{-\alpha}x y^{t}\\
     (\alpha + 1)|x|^{-\alpha}x^{t} y &   (\alpha + 1)^2  |x|^{-2\alpha}yy^{t}\end{pmatrix}}
\newcommand{\norm}[1]{\lVert#1 \rVert}
\newcommand{\ve}{\varepsilon}
\newcommand{\D}{\operatorname{div}}
\newcommand{\G}{\mathscr{G}}

\title[space-time propagation]{On the forward in time propagation of zeros  in fractional heat type problems}

\author{Agnid Banerjee}
\address{Tata Institute of Fundamental Research\\
Centre For Applicable Mathematics \\ Bangalore-560065, India}\email[Agnid Banerjee]{agnidban@gmail.com}

\author{Nicola Garofalo}

\address{Dipartimento d'Ingegneria Civile e Ambientale (DICEA)\\ Universit\`a di Padova\\ Via Marzolo, 9 - 35131 Padova,  Italy}
\vskip 0.2in
\email{nicola.garofalo@unipd.it}

\thanks{A. Banerjee  is supported in part  by Department of Atomic Energy,  Government of India, under
project no.  12-R \& D-TFR-5.01-0520. N. Garofalo is supported in part by a BIRD grant: ``Aspects of nonlocal operators via fine properties of heat kernels", Univ. of Padova, 2022. }

%
%
%

\maketitle

\begin{abstract}
In this short note we prove that if $u$ solves  $(\partial_t - \Delta)^s u = Vu$ in $\R^n_x \times \R_t$, and vanishes to infinite order at a point $(x_0, t_0)$, then $u \equiv 0$ in $\R^n_x \times \R_t$. This sharpens (and completes) our earlier result that proves $u(\cdot, t) \equiv 0$ for $t \leq t_0$ if it vanishes to infinite order at $(x_0, t_0)$. 
\end{abstract}

\section{Introduction and Statement of the main result}
Given a parameter $0<s<1$ we consider the following nonlocal equation in space time $\R^{n+1}= \R^n_x \times \R_t$  
\begin{equation}\label{e0}
(\partial_t - \Delta)^s u = Vu,
\end{equation} 
where $V$ satisfies the following structural assumptions
\begin{equation}\label{vasump}
\begin{cases}
||V||_{C^{1}(\R^{n+1})}    \le K,\ \ \ \ \ \text{if}\ 1/2\le s < 1,
\\
\\
||V||_{C^{2}(\R^{n+1})},\ \ \ \ \   ||<\nabla_x V, x>||_{L^\infty(\R^{n+1})} \le K,\ \ \ \ \  \text{if}\ 0< s < 1/2.
\end{cases}
\end{equation}
The following strong backward uniqueness result was proven in \cite[Theorem 1.2]{BG} (we refer to Section \ref{pr} for the precise definitions and notations).

\noindent \textbf{Theorem A.}\label{BG}
\emph{Let $u\in  \operatorname{Dom}(H^{s})$ be a solution to \eqref{e0}  in $\R^{n+1}$ with $V$ satisfying \eqref{vasump}. If $u$ vanishes to infinite order backward in time at some point $(x_0, t_0)$ in $\R^{n+1}$ in the following sense
\begin{equation}\label{vp}
\underset{Q_{r}(x_0, t_0)}{\operatorname{essup}}\ |u| = O(r^N),\end{equation}
 then $u(\cdot,t) \equiv 0$ for all $t \leq t_0$}.

Theorem \hyperref[T:bbm]{A} was proved by combining a basic monotonicity result of an adjusted Poon type functional for the extension problem associated to the operator $(\partial_t - \Delta)^s$ with a blowup analysis of the so-called Almgren type rescalings of a solution to such an extension problem. Theorem \hyperref[T:bbm]{A} however does not shed any light on the  forward propagation of zeros. I.e., it does not address the question whether $u( \cdot, t) \equiv 0$ for $t>t_0$. In this note we provide an affirmative answer to such a question by an elementary variational argument inspired to some ideas in \cite{DG}. Our main result is as follows.

\begin{thrm}\label{main}
Let $u\in  \operatorname{Dom}(H^{s})$ be a solution to \eqref{e0}  in $\R^{n+1}$ with $V$ satisfying \eqref{vasump}. If $u$ vanishes to infinite order backward in time at some point $(x_0, t_0)$ in $\R^{n+1}$ in the sense of \eqref{vp} above,
 then $u \equiv 0$ in $\R^{n+1}$.  
\end{thrm}

The paper is organized as follows. In Section \ref{pr} we introduce some basic notations  and gather some preliminary results that are relevant to our work. In Section \ref{s:m} we prove Theorem \ref{main}.


\section{Preliminaries}\label{pr}

Given a function $f\in L^1(\Rn)$, we denote by $\hat f$ its Fourier transform defined by
\[
\hat f(\xi) = \mathcal F_{x\to \xi}(f) = \int_{\Rn} e^{-2\pi i<\xi,x>} f(x) dx.
\]
A typical point in $\R^{n+1} = \Rn_x \times \R_t$ will be indicated by $(x,t)$. 
The heat operator in $\R^{n+1}$ will be denoted by $H = \p_t - \Delta_x$. Given a number $s\in (0,1)$ the notation $H^s$ will indicate the fractional power of $H$ that in  \cite[formula (2.1)]{Sam} was defined on a function $f\in \mathscr S(\R^{n+1})$ by the formula
\begin{equation}\label{sHft}
\widehat{H^s f}(\xi,\sigma) = (4\pi^2 |\xi|^2 + 2\pi i \sigma)^s\  \hat f(\xi,\sigma),
\end{equation}
with the understanding that we have chosen the principal branch of the complex function $z\to z^s$. The natural domain for $H^s$ is the  parabolic Sobolev space of fractional order $2s$ defined as follows 
\begin{align}\label{dom}
\mathscr H^{2s} & =  \operatorname{Dom}(H^s)   = \{f\in \mathscr S'(\R^{n+1})\mid f, H^s f \in L^2(\R^{n+1})\}
\\
&  = \{f\in L^2(\R^{n+1})\mid (\xi,\sigma) \to (4\pi^2 |\xi|^2 + 2\pi i \sigma)^s  \hat f(\xi,\sigma)\in L^2(\R^{n+1})\},
\notag
\end{align} 
where the second equality is justified by \eqref{sHft} and Plancherel theorem. 
It turns out that the  definition \eqref{sHft} is equivalent to the one based on Balakrishnan formula (see \cite[(9.63) on p. 285]{Samko})
\begin{equation}\label{balah}
H^s f(x,t) = - \frac{s}{\Gamma(1-s)} \int_0^\infty \frac{1}{\tau^{1+s}} \big(P^H_\tau f(x,t) - f(x,t)\big) d\tau,
\end{equation}
where we have denoted by
\begin{equation}\label{evolutivesemi}
P^H_\tau f(x,t) = \int_{\Rn} G(x-y,\tau) f(y,t-\tau) dy = G(\cdot,\tau) \star f(\cdot,t-\tau)(x)
\end{equation}
the \emph{evolutive semigroup}, see \cite[(9.58) on p. 284]{Samko}, i.e., the solution $u((x,t),\tau) = 
P^H_\tau f(x,t)$ of the Cauchy problem in $\R^{n+1}_{(x,t)}\times \R^+_\tau$
\[
\p_\tau u = \Delta_x u - \p_t u,\ \ \ \ \ \ u((x,t),0) = f(x,t).
\]
We next consider the thick half-space $\R^{n+1}_{(x,t)} \times \R^+_y $. At times it will be convenient to indicate points $(x,t),y)$ in such space by combining the extension variable $y>0$ with $x\in \Rn$, and denote the generic point in the thick space $\R^{n+1}_+ = \Rn_x\times\R^+_y$ with the letter $X=(x,y)$. For notational ease $\nabla U$ and  $\operatorname{div} U$ will respectively refer to the quantities  $\nabla_X U$ and $ \operatorname{div}_X U$.  The partial derivative in $t$ will be denoted by $\p_t U$ and also at times  by $U_t$. The partial derivative $\partial_{x_i} U$  will be denoted by $U_i$. At times,  the partial derivative $\partial_{y} U$  will be denoted by $U_{n+1}$. 

We next introduce the extension problem associated with $H^s = (\partial_t - \Delta)^s$ is as follows. Given a number $a\in (-1,1)$ and a $u:\R^n_x\times \R_t\to \R$ we seek a function $U:\R^n_x\times\R_t\times \R_y^+\to \R$ that satisfies the degenerate Dirichlet problem
\begin{equation}\label{la}
\begin{cases}
\La U \overset{def}{=} \partial_t (y^a U) - \operatorname{div} (y^a \nabla U) = 0,
\\
U((x,t),0) = u(x,t),\ \ \ \ \ \ \ \ \ \ \ (x,t)\in \R^{n+1}.
\end{cases}
\end{equation}
The most basic property of the problem \eqref{la} is that if $s = \frac{1-a}2\in (0,1)$, then for any $u \in \mathscr H^{2s}$ one has in $L^2(\R^{n+1})$
\begin{equation}\label{np}
2^{-a}\frac{\Gamma(\frac{1-a}{2})}{\Gamma(\frac{1+a}{2})} \py U((x,t),0)=  - H^s u(x,t),
\end{equation}
where $\py$ denotes the weighted normal derivative
\begin{equation}\label{nder}
\py U((x,t),0)\overset{def}{=}   \operatorname{lim}_{y \to 0^+}  y^a \partial_y U((x,t),y).
\end{equation}
When $a = 0$, equivalently $s = 1/2$, the problem \eqref{la} was first introduced in \cite{Jr1} by Frank Jones, who in such case also constructed the relevant Poisson kernel and proved \eqref{np}. More recently Nystr\"om and Sande in \cite{NS} and Stinga and Torrea in \cite{ST} have independently extended the results in \cite{Jr1} to all $a\in (-1,1)$. 

In view of \eqref{e0}, \eqref{la}, \eqref{np} and of \cite[Corollary 4.6]{BG} we know that, for a given $u \in \mathscr H^{2s}$ solution to \eqref{e0} (with $V$ satisfying \eqref{vasump}), the extended function $U$ in \eqref{la} is a weak solution of the following problem 
\begin{equation}\label{wk}
\begin{cases}
\La U=0 \ \ \ \ \ \ \ \ \ \ \ \ \ \ \ \ \ \ \ \ \ \ \ \ \ \ \ \ \ \ \text{in}\ \R^{n+1}\times \R^+_y,
\\
U((x,t),0)= u(x,t)\ \ \ \ \ \ \ \ \ \ \ \ \ \ \ \ \text{for}\ (x,t)\in \R^{n+1},
\\
\py U((x,t),0)=  2^{a} \frac{\Gamma(\frac{1+a}{2})}{\Gamma(\frac{1-a}{2})} V(x,t) u(x,t)\ \ \ \ \text{for}\ (x,t)\in \R^{n+1}.
\end{cases}
\end{equation} 
Further, in \cite[Lemma 5.3]{BG} the following regularity result for such weak solutions was proved (for the precise notion of parabolic H\"older spaces $H^\alpha$, we refer to \cite[Chapter 4]{Li}).

\begin{lemma}\label{reg1}
Let $U$  be a weak solution of \eqref{wk} corresponding to $u \in \mathscr{H}^{2s}$ where $V$ satisfies the growth condition in \eqref{vasump}. Then there exists $\alpha>0$ such that one has up to the thin set $\{y=0\}$ 
\[
U_i,\ U_t,\ y^a U_y\ \in\ H^{\alpha}(\overline{\R^{n+1}_+} \times \R),\ \ \ \  i=1,2,..,n.
\]
\end{lemma}

We next recall that it was shown in \cite{Gcm} that given $\phi \in C_0^{\infty}(\R^{n+1}_+)$ the solution of the Cauchy problem with Neumann condition
\begin{align}\label{CN}
\begin{cases}
\La U=0 \hspace{2mm} &\text{in}  \hspace{2mm}\R^{n+1}_+ \times (0,\infty)\\
U(X,0)=\phi(X), \hspace{2mm} &X \in \R^{n+1}_+,\\
\py U(x, 0, t) = 0 \hspace{2mm} & x\in \Rn,\ t \in (0, \infty)
\end{cases}
\end{align}
is given by the formula 
\begin{equation}\label{Ga}
\mathscr P^{(a)}_t \phi(X_1) \overset{def}{=} U(X_1,t) = \int_{\R^{n+1}_+} \phi(X) \G (X_1,X,t) y^a dX,
\end{equation}
where 
\begin{equation}\label{funda}
\G(X_1,X,t) = p(x_1,x,t)\ p^{(a)}(y_1,y,t)
\end{equation}
is the product of the standard Gauss-Weierstrass kernel $p(x_1,x,t) = (4\pi t)^{-\frac n2} e^{-\frac{|x_1-x|^2}{4t}}$ in $\Rn\times\R^+$ with the heat kernel of the Bessel operator $\mathscr B_a = \p_{yy} + \frac ay \p_y$  with Neumann boundary condition in $y=0$ on $(\R^+,y^a dy)$ (reflected Brownian motion)
 \begin{align}\label{fs}
p^{(a)}(y_1,y,t) & =(2t)^{-\frac{a+1}{2}}\left(\frac{y_1 y}{2t}\right)^{\frac{1-a}{2}}I_{\frac{a-1}{2}}\left(\frac{y_1 y}{2t}\right)e^{-\frac{y_1^2+y^2}{4t}}.
\end{align}
In \eqref{fs} we have denoted by $I_{\frac{a-1}{2}}$ the modified Bessel function of the first kind and order $\frac{a-1}{2}$ defined by the series 
\begin{align}\label{besseries}
 I_{\frac{a-1}{2}}(z) = \sum_{k=0}^{\infty}\frac{(z/2)^{\frac{a-1}{2}+2k} }{\Gamma (k+1) \Gamma(k+1+(a-1)/2)}, \hspace{4mm} |z| < \infty,\; |\operatorname{arg} z| < \pi.
\end{align}
For future use we note explicitly that \eqref{funda} and \eqref{fs} imply that for every $x, x_1\in \Rn$ and $t>0$ one has
\begin{equation}\label{neu}
\operatorname{lim}_{y \to 0^+}  y^a \partial_y \G((x,y),(x_1,0),t) =0. 
\end{equation}
It is important for the reader to keep in mind that \eqref{Ga} defines a stochastically complete semigroup (see \cite[Propositions 2.3 and 2.4]{Gcm}), and therefore in particular we have for every $X_1\in \R^{n+1}_+$ and $t>0$ 
\begin{equation}\label{Ga1}
\mathscr P^{(a)}_t 1(X_1) = \int_{\R^{n+1}_+} \G (X_1,X,t) y^a dX = 1,
\end{equation}
and also \begin{equation}\label{delta} \mathscr P^{(a)}_t \phi(X_1) \underset{t\to 0^+}{\longrightarrow} \phi(X_1).\end{equation}

\section{Proof of Theorem \ref{main}}\label{s:m}

\begin{proof}[Proof of Theorem \ref{main}]
Let $u$ be as in the statement of Theorem \ref{main}, and denote by $U$ the corresponding solution to the problem \eqref{wk}. According to the choice of notation in the introduction, in the sequel we write $U(X,t)$, instead of $U((x,t), y)$ whenever convenient. Without loss of generality, we assume that $(x_0, t_0) =(0,0)$. From the above cited Theorem \hyperref[T:bbm]{A} in \cite{BG} it follows that  if \eqref{vp} holds at $(0,0)$, then $U(\cdot, t) =0$ for all $t \leq 0$, and thus in particular
\begin{equation}\label{zero}
U(\cdot,0)\equiv 0.
\end{equation}
Also, to simplify the notation we henceforth indicate with $V(x,t)$ the function $2^{a} \frac{\Gamma(\frac{1+a}{2})}{\Gamma(\frac{1-a}{2})} V(x,t)$ in the right-hand side of the third equation in \eqref{wk}, which thus becomes
\begin{equation}\label{new}
\py U((x,t),0)=  V(x,t) u(x,t)\ \ \ \ \text{for}\ (x,t)\in \R^{n+1}.
\end{equation}
Following an idea in \cite{DG}, for a fixed $X_0 = (x_0,y_0) \in \R^{n+1}_+$ and a number $T >0$ to be subsequently chosen suitably small, we define for $0 < R < \sqrt{T}$
\begin{equation}\label{phi}
\phi(R)= \int_{\R^{n+1}_+}  U^2(X, R^2) \mathscr G(X,X_0,T-R^2) y^a dX.
\end{equation}
Differentiating  \eqref{phi} one has
\begin{align}\label{e1}
& \phi'(R) = 4 R \int   U U_t( X,R^2) \mathscr G y^a - 2R \int   U^2\partial_t \mathscr G y^a.
\end{align} 
The differentiation under the integral sign is justified by the regularity estimates in Lemma \ref{reg1} and by an approximation argument similar to that in Section 6 in \cite{BG}. Using now the equation satisfied by $\G$ we find
\begin{equation}\label{e2}
\phi'(R)= 4R \int  U U_t( X,R^2) \mathscr G y^a - 2R \int U^2 \operatorname{div}(y^a \nabla \G) y^a.
\end{equation}
Integrating by parts in the second integral in \eqref{e2}, and also using \eqref{wk} (which is now \eqref{new}) and \eqref{neu}, we find
\begin{align}\label{e4}
&\phi'(R)= - 4R \int  U \La U  \mathscr G y^a - 4R  \int |\nabla U|^2 \G y^a  - 4R \int_{\{y=0\}} VU^2 \G dx\\
& = - 4R  \int |\nabla U|^2 \G y^a  - 4R \int_{\{y=0\}} VU^2 \G dx.\notag\end{align}
The boundary integral in \eqref{e4} is now estimated as in \cite[formula (3.14)]{ABDG} in the following way
\begin{equation}\label{e5}
\left|4R \int_{\{y=0\}} VU^2 \G dx \right| \leq C_1  R \left( (T-R^2)^{-\frac{1+a}{2}} \int U^2 \G y^a + (T-R^2)^{\frac{1-a}{2}} \int |\nabla U|^2 \G y^a \right),
\end{equation}
where $C_1= C_1(n,a, ||V||_{L^{\infty}})>0$. Since one has trivially
\[
C_1 R (T-R^2)^{\frac{1-a}{2}} \int |\nabla U|^2 \G y^a \leq C_1 R T^{\frac{1-a}{2}} \int |\nabla U|^2 \G y^a,
\]
by choosing $T>0$ sufficiently small we can ensure that
\begin{equation}\label{e8}
C_1 R T^{\frac{1-a}{2}} \int |\nabla U|^2 \G y^a <   4R  \int |\nabla U|^2 \G y^a.
\end{equation}
Using \eqref{e5} and  \eqref{e8}  in  \eqref{e4}, we finally deduce that for $T>0$ sufficiently small and $0<R<\sqrt T$, one has  
\begin{equation}\label{f1}
\phi'(R) \leq  C_1 R (T-R^2)^{-\frac{1+a}{2}} \phi(R).
\end{equation}
It follows from \eqref{f1} that, for the new constant $C = \frac{C_1}{(1-a)}>0$, the function
$$F(R) \overset{def}=  e^{C (T-R^2)^{\frac{1-a}{2}}} \phi(R)$$ is monotonically decreasing. Since \eqref{zero} implies $\phi(0) =0$, we conclude that
\[
F(R) \leq F(0) =  e^{CT^{\frac{1-a}{2}}} \phi(0) =0.
\]
Thus $ F\equiv 0$ on $[0,\sqrt T)$. If we now let $R \to  \sqrt{T}$, and use the Dirac $\delta$ property of $\G$ in \eqref{delta}, we find that
\[
\operatorname{lim}_{R \to \sqrt{T}}   F(R) =  U(X_0, T)^2 =0.
\]
By the arbitrariness of $X_0 \in \R^{n+1}_+$ we infer that $U \equiv 0$ in $\R^{n+1}_+ \times (0, T)$. Repeating the above arguments on successive intervals $(T, 2T)$ and so on, we finally obtain that
$U \equiv 0$ in $\R^{n+1}_+ \times \R$. The desired conclusion now follows since $U(x, 0, t) = u(x,t)$. 

\end{proof}

\end{document}